\documentclass[letterpaper, 10 pt, conference]{ieeeconf}  

\IEEEoverridecommandlockouts                              

\usepackage{graphics} 
\usepackage{epsfig} 
\usepackage{epstopdf}
\usepackage{amsmath} 
\usepackage{amssymb}  
\usepackage{amsfonts}
\usepackage{subfigure}
\usepackage{nicefrac}

\usepackage{xspace}
\usepackage{soul}
\usepackage{mathtools}
\usepackage{makecell}
\usepackage{todonotes}
\usepackage{nicefrac}

\newcommand{\ra}{\rightarrow}

\newcommand{\ie}{\unskip, i.\,e.,\xspace}
\newcommand{\eg}{\unskip, e.\,g.,\xspace}

\newcommand{\sut}{\text{s.\,t.\,}}

\newcommand{\nrm}[1]{\left\lVert#1\right\rVert}

\newcommand{\abs}[1]{\left\lvert#1\right\rvert}
\newcommand{\scal}[1]{\left\langle#1\right\rangle}
\newcommand{\tr}[1]{\ensuremath{\text{tr}}\left(#1\right)}
\newcommand{\E}[1]{\mathbb E\left[#1\right]}
\newcommand{\PP}[1]{\mathbb P\left[#1\right]}

\newcommand{\N}{\ensuremath{\mathbb{N}}}

\newcommand{\R}{\ensuremath{\mathbb{R}}}

\newcommand{\X}{\ensuremath{\mathbb{X}}}

\newcommand{\U}{\ensuremath{\mathbb{U}}}

\newcommand*\diff{\mathop{}\!\mathrm{d}}
\newcommand{\eps}{\ensuremath{\varepsilon}}

\newcommand{\spc}{\ensuremath{\,\,}}

\newcommand{\ball}{\ensuremath{\mathcal B}}

\newcommand{\co}{\ensuremath{\overline{\text{co}}}}

\DeclareMathOperator*{\argmin}{arg\,min}

\definecolor{dgreen}{rgb}{0.0, 0.5, 0.0}

\newtheorem{dfn}{Definition}

\newtheorem{thm}{Theorem}
\newtheorem{crl}{Corollary}
\newtheorem{rem}{Remark}

\newcommand{\SH}{S\&H\xspace}

\newcommand{\D}{\ensuremath{\mathcal{D}}}
\newcommand{\A}{\ensuremath{\mathcal{A}}}
\newcommand{\K}{\ensuremath{\mathcal{K}}}

\newcommand{\lip}[2]{\ensuremath{\text{Lip}_{#2}\left(#1\right)}}

\newcommand{\bigo}[1]{\ensuremath{\mathcal O \left(#1\right)}}

\newcommand{\Var}[1]{\mathbb V\left[#1\right]}

\newcommand{\as}{a.~s.\xspace}

\title{On stochastic stabilization of sampled systems}

\author{Pavel Osinenko, Grigory Yaremenko
\thanks{Email: \texttt{p.osinenko@yandex.ru}.
}
}

\begin{document}

\maketitle
\thispagestyle{empty}
\pagestyle{empty}

\begin{abstract}
This paper addresses stochastic stabilization in case where implementation of control policies is digital, i. e., when the dynamical system is treated continuous, whereas the control actions are held constant in predefined time steps.
In such a setup, special attention should be paid to the sample-to-sample behavior of the involved Lyapunov function.
This paper extends on the stochastic stability results specifically to address for the sample-and-hold mode.
We show that if a Markov policy stabilizes the system in a suitable sense, then it also practically stabilizes it in the sample-and-hold sense.
This establishes a bridge from an idealized continuous application of the policy to its digital implementation.
The central result applies to dynamical systems described by stochastic differential equations driven by the standard Brownian motion.
Generalizations are discussed, including the case of non-smooth Lyapunov functions for systems driven by bounded noise.
A brief overview of bounded noise models is given.
\end{abstract}


\section{Introduction}\label{sec:intro}

Stochastic systems are an important abstraction of practical objects and processes such as stock markets, noisy devices, chemical reactions etc.
Studying stability of such systems is challenging as compared to the deterministic case as many results from the real analysis do not generalize to the stochastic case.
A major advancement of the classical results on stability of stochastic differential equations (SDE) due to Khasminskii \cite{khasminskii201-stochastic}, Kushner \cite{Kushner1965stochastic-stability}, and Mao \cite{mao1991-stability} was made by Deng et al. \cite{Deng2001-stochastic-stab-noise} who adopted the prominent $\K$-function techniques of Khalil \cite{Khalil1996-nonlin-sys} to the stochastic case.
However, Deng et al. \cite{Deng2001-stochastic-stab-noise}, while studying asymptotic stability in probability, did not provide explicit convergence rates.
Similar principles to \cite{Deng2001-stochastic-stab-noise} found wide application in various stochastic stability analyses, including discrete systems \cite{McAllister2003-stochastic}, cascaded systems \cite{Liu2008-stochastic}, delayed systems \cite{Liu2011-stochastic}, systems with input saturation \cite{Li2017-Stochastic}, systems with state-dependent switching \cite{Wu2013-stochastic},
hybrid systems \cite{Teel2014-stochastic-stability} etc.
Results on practical stochastic stability are known (see \eg some recent ones in \cite{caraballo2015-practical-stochastic,Qin2020-stochastic,Do2020-stochastic}), whereas relatively few results on stochastic stabilization by sampled control are known and they address specific contexts \eg based on approximate discrete-time models \cite{fu2016sampled,yu2018sampled}.
The of this work is to study practical stabilization of stochastic dynamical systems of the class
\begin{equation}
	\label{eqn:sys}
	\diff  X_t = f(X_t, U_t) \diff t + \sigma (X_t, U_t) \diff W_t, X_0 = x_0 \text{~a.~s.},	
\end{equation}
where $\{X_t\}_t, \{U_t\}_t$ are the state and, respectively, control stochastic processes, $\R^n$-, respectively, $\R^m$-valued; $f: \R^n \times \R^m \ra \R^n, \sigma: \R^n \times \R^m \ra \R^{n \times d}$;
$\{W_t\}_t$ is the standard $d$-dimensional Brownian motion.
The technical goal of this work is to study stabilization of \eqref{eqn:sys} by Markov control policies in the so-called sample-and-hold sense (\SH) \ie where control actions are held constant during time steps of pre-defined constant length.


\textbf{Contribution.} A \SH-analysis technique of Clarke et al. \cite{Clarke2011-discont-stabilization} for practical stabilization is generalized to the stochastic case, which presents a major extension of the previous results.
To this end, elements of the stochastic process theory are employed.
We show in Theorem \ref{thm:stoch-stab} that if a Markov policy stabilizes a given system in mean up to a certain best limit, which depends only on the noise characteristics, then it also practically stabilizes the system up to the same best limit, when implemented in the \SH-mode, whenever certain conditions on the system and the related Lyapunov function hold.
Generalizations to the case of a non-smooth control Lyapunov functions are addressed in Section \ref{sec:applications} (see, in particular, Theorem \ref{thm:stoch-stab-nonsmooth}).

\textbf{Notation}: capital letters denote random variables or, provided with a time index, values of stochastic processes, unless specified otherwise.
Probability measure is denoted $\mathbb P$, expected value operator -- $\mathbb E$, variance -- $\mathbb V$.
A closed ball with a radius $r$ centered at $x$ is denoted $\ball_r(x)$, or just $\ball_r$ if $x=0$.
Class kappa and kappa-infinity functions are denoted $\K, \K_\infty$.
A Lipschitz constant of a function $L$ on a ball $\ball_R$ is denoted $\lip L R$.
The notation $a \land b$, $a \lor b$ for some numbers $a, b$ is a shorthand for $\min\{a, b\}$, respectively, $\max\{a, b\}$.
The gradient vectors are treated as row vectors.
Scalar product is denoted as $\scal{\bullet, \bullet}$. 

\textbf{Abbreviations}: ``sample-and-hold'': \SH. ``stochastic differential equation'': SDE. ``Almost surely'': \as

\section{Preliminaries}\label{sec:prelim}


We start with assuming that the strong solutions (also refereed to as ``trajectories'' throughout) are adapted to the augmented filtration $\{\mathcal F_t\}$ generated by the Brownian motion $\{W_t\}_t$.
Denote $f^\mu := f(x, \mu(x)), \sigma^\mu := \sigma(x, \mu(x))$ for a Markov policy $\mu$.
The generator of the SDE of \eqref{eqn:sys}, for a smooth function $L$, is defined as follows:
\begin{equation}
	\label{eqn:stoch-generator}
	\begin{aligned}
		\A^\mu L(x) = \nabla L f^\mu(x) + \frac 1 2 \tr{(\sigma^\mu(x))^\top \nabla^2 L(x) \sigma^\mu(x)},		
	\end{aligned}
\end{equation}
where $\nabla L$ is the gradient vector and $\nabla^2 L$ is the Hessian \ie the matrix of second-order derivatives.
In the \SH mode, the above stochastic system \eqref{eqn:sys} reads: 
\begin{equation}
	\label{eqn:sys-SH}
	\begin{aligned}
		\diff  X_t & = f(X_t, U^\delta_t) \diff t + \sigma (X_t, U^\delta_t) \diff W_t, X_0 = x_0 \spc \as \\
		U^\delta_t & \equiv U_k, t \in [k \delta, (k+1) \delta],
	\end{aligned}
\end{equation}
where $\delta$ refers to the sampling step size (or, briefly, sampling time).
In particular, applying a Markov policy $\mu$ to \eqref{eqn:sys} in \SH mode means $U_k = \mu(X_{k \delta})$
We will also use the notation $\mu^\delta$ in this case. 
Denote, for any $R$, $\bar f^\mu_R := \sup_{x \in \ball_R} \nrm{f^\mu(x)}, \bar \sigma^\mu_R := \sup_{x \in \ball_R} \nrm{\sigma^\mu(x)}$, and $b_R := \sup_{x \in \ball_R} \nrm{\nabla L(x)}, b'_R := \sup_{x \in \ball_R} \nrm{\nabla^2 L(x)}$.
Define an operator $\Gamma^\mu_R L := \bar f^\mu_{R} \big( \lip{\nabla L}{R} \bar f^\mu_{R} + b_{R} \lip{f}{R} + \bar \sigma_{R} b'_{R} \lip{\sigma}{R} + \frac 1 2 (\bar \sigma_{R}^{\mu})^2 \lip{\nabla^2 L}{R} \big)$ where $\lip{f}{R}, \lip{\sigma}{R}$ are Lipschitz constants in the first arguments.
Now, proceed to the necessary definitions for stability.

\begin{dfn}[Stability in probability]
	\label{dfn:stab-prob}
	The origin of the system \eqref{eqn:sys} is said to be stable in probability if $\exists \kappa \in \K, \lim_{\eps \ra \infty} \kappa(\eps)=1  \spc \forall \eps>0 \spc \exists A,c>0$, \sut
	\begin{equation}
		\label{eqn:stab-prob}
		 x_0 \in \ball_A \implies \forall t \ge 0 \spc \PP{X_t \in \ball_{A+c \eps}} \ge \kappa(\eps).
	\end{equation}
\end{dfn}



\begin{dfn}[$(r,R)$-convergence in mean]
	\label{dfn:conv-mean}
	A trajectory $\{X_t\}_t$ of a stochastic system \eqref{eqn:sys} is said to converge in mean if the following conditions hold:
	\begin{equation}
		\label{eqn:conv-mean}
		\begin{aligned}
			& x_0 \in \ball_R \implies \limsup_{t \ra \infty} \E{\nrm{X_t}} \le r.
		\end{aligned}
	\end{equation}
	Moreover, the reaching time $T_{(r, R)} := \inf_{t \ge 0}\{ \E{\nrm{X_t}} \le r\}$ depends uniformly on $r, R$.
\end{dfn}

\begin{rem}
	The $(r,R)$-convergence in mean can be understood as a kind of practical stability: provided that the initial state was within some starting ball \as, the mean trajectory eventually enters a target ball of radius $r$ and stays there forever.
\end{rem}

\begin{rem}
	Definition \ref{dfn:conv-mean} differs from the usual convergence in mean of the form $\lim_{t \ra \infty} \E{\nrm{X_t}} = 0$ in the sense that we encode the numbers $r, R$ because the main result in the said section is semi-global (see also Remark \ref{rem:semi-glob} below): given a starting ball of radius $R$, a target one of radius $r$, there is a bound on sampling time \sut the sampled policy stabilizes the trajectory into $\ball_r$ in mean.	
	It is, in general, impossible to guarantee convergence in mean to zero as in the standard definition mentioned above.
\end{rem}

The next definition is the central tool for stabilization, namely, a nominal Lyapunov pair consisting of a Lyapunov function $L$ and a Markov policy $\mu$.
The key result of this work lies in investigation of the system behavior under \textit{sampled} policy, that renders the control actions $U_t \equiv \mu(X_{k\delta}), t \in [k \delta, (k+1) \delta]$.

\begin{dfn}[Stochastic Lyapunov pair]
	\label{dfn:stoch-L-pair}
	A stochastic Lyapunov pair $(L, \mu)$ is for the system $\eqref{eqn:sys}$ a pair of functions if:
	$L \in \mathcal C^2$;
	there exist $\bar \alpha_1, \bar \alpha_2 > 0$, $\alpha_3 \in \K_\infty$ with $\alpha_3 \circ \sqrt{\nrm{x}}$ convex; 
	there exists $\alpha_4 \in \K_\infty \spc \forall R \spc \Gamma^\mu_R L \le \alpha_4(R)$ \sut $\alpha_4 \circ \sqrt{R}$ is concave;	
	(monotone condition cf. \cite{mao2007stochastic,yu2018sampled,fu2016sampled}) there exists $K > 0, K_{\mu} \spc \sut \forall x, \mu(x) \in \ball_{K_\mu} \land \forall x, u \in \ball_{K_\mu} \spc x^\top f(x, u) + \tfrac 1 2 \nrm{\sigma(x, u)}^2 \le K (1 + \nrm{x}^2)$; 
	the following properties hold:
	\begin{align}
		\label{eqn:LF-alphas12}
		& \forall x \spc \bar \alpha_1 \nrm{x}^2 \leq L(x) \leq \bar \alpha_2 \nrm{x}^2 \\
		\label{eqn:LF-alpha3}
		& \forall x \spc \A^\mu L(x) \le - \alpha_3(\nrm{x}) + \bar \Sigma, \bar \Sigma > 0.
	\end{align}
\end{dfn}

\begin{rem}
	The number $\bar \Sigma > 0$ in Definition \ref{dfn:stoch-L-pair} is related to the noise and differentiates the stochastic case from a deterministic one, which typically possesses a decay condition of the kind $\scal{\nabla L(x), f(x, \mu(x))} \le - \alpha_3(\nrm{x})$.
\end{rem}

\begin{rem}
	In general, for a function $\varphi$, the Jensen's gap \ie $\E{\varphi(X)} - \varphi( \E{X} )$ can be arbitrarily large.
	To relate various expected values in the analysis of Theorem \ref{thm:stoch-stab}, the Jensen's inequality has to be utilized whence the stated conditions.  
	Notice \eg \cite[Lemma~2.1]{Reif1999-EKF-stab} also used quadratic bounding functions of $L$.
	A condition $\A^\mu L(x) = - c L + \bar \Sigma, c > 0$ (cf. \cite[Theorem~4.1]{Deng2001-stochastic-stab-noise}) also fits the assumptions since $- c L \le - c \bar \alpha_1 \nrm{x}^2$ and $\alpha_3$ is thus effectively $c \bar \alpha_1 \nrm{x}^2$ and so $\alpha_3 \circ \sqrt{\nrm{x}}$ is convex.
	The function $\alpha_4 \circ \sqrt{\nrm{x}}$ being concave is satisfied if \eg $L$ is quadratic, $f^\mu$ is Lipschitz and of linear growth, $\sigma^\mu$ is Lipschitz and bounded.
	This condition may be seen as restrictive, although \eg linear growth is often assumed in stochastic stabilization in mean (see, for instance, \cite{li2021stochastic,lan2017global,yang2011mean,fu2016sampled,yu2018sampled}).
	The monotone condition is in the style of \cite{mao2007stochastic} and is weaker than in related works on sampled stochastic stabilization (see \eg \cite{yu2018sampled,fu2016sampled}), whereas it should be noted that universal formulas with bounded controls are known \cite{Lin1991-stabilization-bounded-controls}.
	Furthermore, this condition will secure global existence of strong solutions \cite{mao2007stochastic}, which is unavoidable in case of \SH mode, and this is in contrast to the ``standard'' Lyapunov techniques in stochastic systems \cite{khasminskii201-stochastic}.
	The reason is that, in the latter, decay of the subject Lyapunov function is ensured for all times, whereas in the herein considered case, there are necessarily time intervals in which the said decay cannot be guaranteed.
	In Section \ref{sec:applications}, we will consider stochastic systems driven by bounded noise, which is physically meaningful, and where the stated conditions can be discarded.
%
%
%
%
\end{rem}

	


\begin{dfn}[Practical stochastic stabilization in mean]
	\label{dfn:pract-stab}
	A Markov policy $\mu$ is said to practically stochastically stabilize \eqref{eqn:sys} for any $\eps > 0$ any $0 < r < R$, there exists $\bar \delta > 0$ \sut for any \SH-mode \eqref{eqn:sys-SH} with $\delta \le \bar \delta$ the following holds:
	\begin{enumerate}
		\item the origin of the system is stable in probability;
		\item trajectories of the system $(r \lor \rho,R)$-converge in mean, where $\rho$ depends on the noise property $\bar \Sigma$.
	\end{enumerate}
\end{dfn}

\begin{rem}
	\label{rem:semi-glob}
	Definition \ref{dfn:stab-prob} bears a semi-global character: the required bound on the sampling time depends on the starting and target ball radii, although they are arbitrary.
\end{rem}

\begin{rem}
	\label{rem:rho}
	The reason for the presence of $\rho$ is due to $\bar \Sigma$, the noise-related term, in the decay condition \eqref{eqn:LF-alpha3}.
	In particular, it may be defined as $\rho = \inf_{\rho'}\{ \inf_{r' \ge  \sqrt{\tfrac{\bar \alpha_1}{2 \bar \alpha_2}}r \lor \rho'  } (\alpha_3(r') - \bar \Sigma) \ge \bar \alpha_3 \} $.
\end{rem}

\section{Main theorem} \label{sec:main}

\begin{thm}[Practical stochastic stabilization]
	\label{thm:stoch-stab}
	Consider a stochastic system \eqref{eqn:sys}.
	Suppose there exists a stochastic Lyapunov pair $(L, \mu)$.
	Then, $\mu$ practically stochastically stabilizes \eqref{eqn:sys} in mean.
\end{thm}

\begin{proof}
	The proof is organized into two parts: first, necessary parameters are determined; second, a \SH-analysis is conducted on the mean Lyapunov function $L$.
	
	\textbf{Settings}.
	Let $0 < r < R$ be given.
	Let $X_t$ be an arbitrary $\mathcal F_t$-measurable random variable and denote $L_t := L(X_t)$ for brevity. 
	From \eqref{eqn:LF-alphas12}, and the Jensen's inequality, observe that
	\begin{equation}
		\label{eqn:alphas1-inverse-E}
		\E{\nrm{X_t}} \le \sqrt{ \nicefrac{1}{\bar \alpha_1} \E{L_t} }.
	\end{equation}	
	Let $\bar L := \sup_{x \in \ball_R} L(x)$ and $R^* := \sqrt{ \nicefrac{1}{\bar \alpha_1} \bar L} $.
	It follows that
	\begin{equation}
		\label{eqn:when-inside-Rstar}
		\E{L_t} \le \bar L \implies \E{\nrm{X_t}} \le R^*.
	\end{equation}
	Define $\rho := \inf_{\rho'}\{ \inf_{r' \ge  \sqrt{\tfrac{\bar \alpha_1}{2 \bar \alpha_2}}r \lor \rho'  } (\alpha_3(r') - \bar \Sigma) \ge \bar \alpha_3 \}$ for a given $\bar \alpha_3$ which is always possible since $\bar \Sigma$ is finite and $\alpha_3$ is $\K_\infty$.
	Denote $l^* := \bar \alpha_1 (r \lor \rho)^2$.
	From the assumption that $\alpha_3 \circ \sqrt{\nrm{x}}$ is convex, use the Jensen's inequality to deduce
	\begin{equation}
		\label{eqn:when-min-decay}
		\E{L_t} \ge \frac{l^*}{2} \implies \E{\alpha_3(\nrm{X_t})} \ge \bar \alpha_3.
	\end{equation}
	Without loss of generality, assume that the starting ball is bigger than $\ball_\rho$.
	The idea of the above settings is, we take a smaller level of the Lyapunov function, namely, $\nicefrac{l^*}{2}$ than the one we can relate to the target ball via $\E{L_t} \le l^* \implies \E{\nrm{x}} \le r$.
	Once the mean Lyapunov function reaches the level $l^*$, it will stay within it from there on.
	The mean Lyapunov function will have guaranteed decay everywhere beyond the level $\nicefrac{l^*}{2}$.
	Notice that due to the presence of $\bar \Sigma$ in \eqref{eqn:LF-alpha3}, the level $l^*$ may not be too small, for otherwise $\bar \alpha_3$ may fail to be negative.
	That is the reason to take $\ball_\rho$ containing $\{ x \in \R^n : \alpha_3(\nrm{x}) \le 2 \bar \Sigma \}$, in particular.
	Proceed now to the \SH analysis, during which necessary bounds on $\bar \delta$ will be determined. 
	
	\textbf{\SH analysis.} Now, proceed to apply the Markov policy $\mu$ in the \SH-mode \eqref{eqn:sys-SH} and check the sampled mean Lyapunov function values.
	Consider a time step $k$.
	
	\textit{Case 1}: $\E{L_{k\delta}} \ge \tfrac{l^*}{2}$.
	
	Due to the monotone condition, the trajectory $\{X_t\}_t$ of \eqref{eqn:sys-SH} exists on the entire $[k\delta, (k+1)\delta]$.
	Define $T_{\mathcal R} : = \inf_{t \ge k\delta } \{ \nrm{X_t} \ge \mathcal R \}$, $\bar T := \inf_{t \ge k\delta } \{ \E{L_t} \ge \bar L\}$, $t_{\mathcal R} : = T_{\mathcal R} \land \bar T \land t, t \in [k\delta, (k+1)\delta]$ and assume that $\E{L_{k \delta}} \le \bar L$.
	Applying the It\=o rule to $L_{t_{\mathcal R}}$, we have
	\begin{equation}
		\label{eqn:Lt-Ito}
		\begin{aligned}
			& L_{t_{\mathcal R}} = L_{k \delta} + \\
			& \int \limits_{k\delta}^{t_{\mathcal R}} \A^{\mu^\delta} L(X_{\tau}) \diff \tau + 
			\int \limits_{k\delta}^{t_{\mathcal R}} \nabla L(X_t) \sigma (X_\tau, \mu(X_{k \delta})) \diff B_\tau.
		\end{aligned}
	\end{equation}		
	Denote the first integral $I_1$ and the second, respectively, $I_2$. 
	Now, apply the stochastic mean value theorem to $I_1$ to yield, for some $T' \in [k\delta, t_{\mathcal R}]$:
		\begin{equation}
			\label{eqn:integr-AL-MVT}
			\begin{aligned}
				I_1 \le & \A^{\mu^\delta} L(X_{T'}) ( t \land T_{\mathcal R} - k\delta ).
			\end{aligned}
		\end{equation}
	Then, using that $L \in \mathcal C^2$, and the bound $\nrm{X_{t_{\mathcal R}} - X_{k \delta}} \le \bar f^{\mu^\delta}_{\mathcal R} (t_{\mathcal R} - k \delta)$, deduce
	\begin{equation}
		\label{eqn:AL-MVT-bounds}
		\begin{aligned}
			& \A^{\mu^\delta} L(X_{T'}) ( t_{\mathcal R} - k\delta ) \le - \bar \alpha_3 ( t_{\mathcal R} - k\delta ) + \\
			& ( t_{\mathcal R} - k\delta )^2 \sup_{\tau' \in [k \delta, t_{\mathcal R}]} \bigg( \bar f^{\mu^\delta}_{ \nrm{X_{\tau'}} } \Big( \lip{\nabla L}{ \nrm{X_{\tau'}} } \bar f^{\mu^\delta}_{ \nrm{X_{\tau'}} } \\
			& + b_{ \nrm{X_{\tau'}} } \lip{f}{ \nrm{X_{\tau'}} } + \bar \sigma_{ \nrm{X_{\tau'}} } b'_{ \nrm{X_{\tau'}} } \lip{\sigma}{ \nrm{X_{\tau'}} } + \\
			& \tfrac 1 2 (\bar \sigma^{\mu^\delta})^2 \lip{\nabla^2 L}{ \nrm{X_{\tau'}} } \Big) \bigg).
		\end{aligned}
	\end{equation}
	The second summand in the right-hand side of \eqref{eqn:AL-MVT-bounds} is $( t_{\mathcal R} - k\delta )^2 \sup_{\tau' \in [k \delta, t_{\mathcal R}]} \Gamma^{\mu^\delta}_{\nrm{X_{\tau'}}} L$.
	Notice that $\A^{\mu^\delta} L(X_{T'}) ( t_{\mathcal R} - k\delta ) \le 0$ whenever $-\bar \alpha_3 + \delta \sup_{\tau' \in [k \delta, t_{\mathcal R}]} \Gamma^{\mu^\delta}_{\nrm{X_{\tau'}}} L \le 0$ since $( t_{\mathcal R} - k\delta ) \le \delta$ \as
	Therefore, set $\bar \delta \le \tfrac{\bar \alpha_3 \bar \alpha_1}{2 \alpha_4 \left( \sqrt{\bar L} \right)}$.
	Furthermore, the global existence of solutions implies $T_{\mathcal R} \ra \infty$ as $\mathcal R \ra \infty$
	Thus, taking this limit, the Fatou's lemma and the Jensen's inequality yields
	\begin{equation}
		\label{eqn:Fatou-GammaR}
		\begin{aligned}
			& \E{\lim_{\mathcal R \ra \infty} \sup_{\tau' \in [k \delta, t_{\mathcal R}]} \Gamma^{\mu^\delta}_{\nrm{X_{\tau'}}} L} \le \tfrac{1}{\bar \alpha_1} \alpha_4 \left(\sqrt{\bar L} \right). 
		\end{aligned}
	\end{equation}	
	Notice that $\E{I_2} = 0$ due to the martingale property of the It\=o integral (up to $t_{\mathcal R}$ with fixed $\mathcal R$, cf. \cite{Deng2001-stochastic-stab-noise}).
	Now, applying the expected value operator to \eqref{eqn:Lt-Ito} yields 
	\begin{equation}
		\label{eqn:Lt-decay}
		\begin{aligned}
			\E{L_{t \land \bar T}} - \E{L_{k\delta}} \le & - \frac{\bar \alpha_3}{2} ( t \land \bar T - k\delta ).
		\end{aligned}
	\end{equation}		
	Since this holds for any $t$ up to $\bar T$, we have $\E{L_{t \land \bar T}} \le \bar L$, whence $\bar T$ has to be not less than $(k+1) \delta$ (cf. \cite{Clarke1997-stabilization}).
	Recalling that $\E{L_t} \le \bar L$ implies $\E{\nrm{X_t}} \le R^*$, boundedness of the trajectory in mean by $R^*$ on $[k\delta, (k+1)\delta]$ follows.		
	
	\textit{Case 2}: $\E{L_{k \delta}} \le \tfrac{3 l^*}{4}$.
	
	First, evidently, $\E{L_{k \delta}} \le \tfrac{3 l^*}{4}$ implies $\E{\nrm{X_{k \delta}}} \le r \lor \rho$.
	Notice that $\E{\nrm{X}} \le \nicefrac{1}{\bar \alpha_1} L_t$.
	From the Grönwall's inequality \cite{mao2007stochastic}, for $t \in [k \delta, (k+1)\delta]$,
	\begin{equation}
		\label{eqn:X-Gronwall}
		\E{\nrm{X_t}^2} \le \sqrt{1 + \E{\nrm{X_{k \delta}}}^2} e^{2 K \delta}
	\end{equation}
	Then,
	\begin{equation}
		\label{eqn:X-Gronwall-to-L}
		\begin{aligned}
			\E{L_t} & \le E{\nrm{X_t}^2} \le \sqrt{1 + \E{\nrm{X_{k \delta}}}^2} e^{2 K \delta} \\
			& \le \sqrt{1 + \nicefrac{1}{\bar \alpha_1^2}\E{L_{k\delta}}^2} e^{2 K \delta}.
		\end{aligned}
	\end{equation}	
	Using the fact that $\E{L_{k \delta}} \le \tfrac{3 l^*}{4}$, one can deduce a second bound on $\bar \delta$ so that, if $\delta \le \bar \delta$, then $\E{\abs{L_t - L_{k\delta}}} \le \tfrac{l^*}{4} \delta$.	
	Thus, once $\E{L_{k \delta}} \le \tfrac{3 l^*}{4}$, $\E{L_{k \delta}} \le l^*$ for all subsequent $k$.
	Combining this with the claim from the previous case, we conclude that the trajectory exists for any $t \ge 0$ and satisfies $\E{\nrm{X_t}} \le R^*$.
	Notice that switching of the control at the nodes $k \delta$ poses no problem since $f^\mu, \sigma^\mu$ need only be measurable in $t$ \cite{oksendal2003stochastic}.
	
	Now, observe that since $\E{L_t} \le \bar L, \forall t \ge 0$, $\E{\nrm{X_t}^2} \le \nicefrac{1}{\bar \alpha_1} \bar L$.
	Combining this with the fact that $\E{\nrm{X_t}}^2 \le {R^*}^2$, deduce that $\Var{\nrm{X_t}}$ is bounded by $s^2:=\abs{\nicefrac{1}{\bar \alpha_1} \bar L - {R^*}^2}$.
	Then, observe that $\forall \eps>0 \spc \nrm{X_t} \ge R^* + \eps s \implies \nrm{X_t} \ge \E{\nrm{X_t}} + \eps s$ since $\forall t \ge 0 \spc \E{\nrm{X_t}} \le R^*$.
	Therefore, $\PP{\nrm{X_t} \ge R^* + \eps s} \le \PP{\nrm{X_t} \ge \E{\nrm{X_t}} + \eps s}$.
	By the one-tailed Chebyshev's inequality, $\PP{\nrm{X_t} \ge \E{\nrm{X_t}} + \eps s} \le \nicefrac{1}{1+\eps^2}$.
	So, $\PP{\nrm{X_t} \le R^* + \eps s} \ge \nicefrac{\eps^2}{1+\eps^2}$.
	Matching this with Definition \ref{dfn:stab-prob}, set, $\forall \eps >0$, $A :\equiv R^*, c :\equiv s$ and $\kappa(\eps) := \nicefrac{\eps^2}{1+\eps^2}$. 
	To conclude, we have $(r \lor \rho,R)$-convergence of the trajectory in mean.
	Derive the bound on the reaching time independent of $\delta$ as follows:
	\begin{equation}
		\label{eqn:raching-time}
		T_{(r,R)} \le \tfrac{4 \bar L - 3 l^*}{2 \bar \alpha_3}.
	\end{equation}	
\end{proof}

%

\begin{crl}
	\label{crl:determ2stoch}
	Suppose that \eqref{eqn:LF-alpha3} have the form
	\begin{equation}
		\label{eqn:smooth-CLF-decay-noiseless}
		\forall x \spc \nabla L(x) f(x,\mu(x)) \le - \alpha_3(\nrm{x}).
	\end{equation}
	In other words, $(L, \mu)$ is a Lyapunov pair for the noiseless system $\dot x = f(x, u)$.
	If it holds that
	\begin{equation}
		\label{eqn:sigma2nabla2-growth}
		\begin{aligned}
			& \exists \tilde r > 0 \spc \forall x \notin \ball_{\tilde r} \\
			& \alpha_3(\nrm{x}) \ge \tfrac 1 2 \nrm{\sigma^\top(x,\mu(x)) \nabla^2 L(x) \sigma(x,\mu(x))}
		\end{aligned}
	\end{equation}
	then, $\mu$ practically stochastically stabilizes \eqref{eqn:sys} in mean.
	In particular, \eqref{eqn:sigma2nabla2-growth} holds if $\nrm{\sigma}$ is uniformly bounded and $\nrm{\nabla^2 L(x)}$ has a growth rate lower than that of $\alpha_3$ everywhere except for a vicinity of the origin.
\end{crl}

\begin{proof}
	Set
	\[
		\bar \Sigma := \tfrac 1 2 \sup_{x \in \ball_{\tilde r}} \nrm{\sigma^\top(x,\mu(x)) \nabla^2 L(x) \sigma(x,\mu(x))}.
	\]
	Define, for all $x$,
	\[
		\hat \alpha_3(\nrm{x}) \hspace{-3pt} := \hspace{-2pt} \alpha_3(\nrm{x}) \hspace{-2pt} - \hspace{-2pt} \tfrac 1 2 \tr{\sigma^\top \hspace{-2pt} (x,\mu(x)) \nabla^2 L(x) \sigma(x,\mu(x))}.
	\]
	Then, for all $x$, it holds that
	\begin{equation*}
		\begin{aligned}
			& - \alpha_3(\nrm{x}) \hspace{-2pt} + \hspace{-2pt} \tfrac 1 2 \tr{\sigma^\top(x,\mu(x)) \nabla^2 L(x) \sigma(x,\mu(x))} \le \\
			& - \hat \alpha_3(\nrm{x}) + \bar \Sigma
		\end{aligned}
	\end{equation*}
	which ensures a decay condition of the kind \eqref{eqn:LF-alpha3} with the generator $\A$ involved.
\end{proof}

\begin{rem}
	The growth condition \eqref{eqn:sigma2nabla2-growth} in Corollary \ref{crl:determ2stoch} may be justified as follows.
	Roughly speaking, taking derivatives decreases the growth rate.
	That is, one would normally expect that, outside some vicinity of the origin, $\nrm{\nabla^2 L(x)}$ grows slower than $\nrm{\nabla L(x)}$.
	Such is the case when $L$ is \eg polynomial.
	The diffusion function $\sigma$ describes the noise magnification depending on the state and control action.
	It may be justified in some applications to assume this term to be bounded uniformly in $x, u$.
	All in all, Corollary \ref{crl:determ2stoch} gives a particular hint on transferring a Lyapunov pair from a nominal, noiseless, to a noisy system.
\end{rem}


\section{Generalizations and applications}\label{sec:applications}

The central result in Section \ref{sec:main} applies to stochastic dynamical systems provided with a Lyapunov pair $(L,\mu)$ describing a Lyapunov function $L$ for the closed-loop under a Markov policy $\mu$.
There is no straightforward extension of Theorem \ref{thm:stoch-stab} to the case of a general control Lyapunov function (CLF) which may have a decay condition of the kind
\begin{equation}
	\label{eqn:smooth-CLF-decay}
	\begin{aligned}
		& \forall \text{ compact } \X \spc \exists \text{ compact } \U_\X \\
		& \forall x \in \X \spc \inf_{u \in \U_\X} \A^u L(x) \le - \alpha_3(\nrm{x}) + \bar \Sigma.	
	\end{aligned}
\end{equation}
The reason is that there is no uniform bound on the realizations of $\{W_t\}_t$.
Consequently, there does not exist a $\bar x > 0$ \sut $\forall t \ge 0 \spc \{X_t\}_t \le \bar x$ \as
What it means practically is that there is no way to determine a uniform bound on the control actions.
It should be noted here that if the CLF is smooth, $\mu$ may be computed analytically by a universal formula \cite{Sontag1989-formula,Florchinger1994-Lyapunov-like-stochastic-stability}, with a variant yielding bounded controls \cite{Lin1991-stabilization-bounded-controls}.
However, most CLFs are actually non-smooth due to a general failure of existence of continuous feedback laws \cite{Clarke2011-discont-stabilization}.
A non-smooth CLF may have a decay condition of the kind
\begin{equation}
	\label{eqn:nonsmooth-CLF-decay}
	\begin{aligned}
		& \forall \text{ compact } \X \spc \exists \text{ compact } \U_\X \\
		& \forall x \in \X \spc \inf_{\nu \in \co(f(x, \U_\X))} \D_\nu L(x) \le - \alpha_3(\nrm{x}) + \bar \Sigma,	
	\end{aligned}
\end{equation}
where $\D_\nu$ is a suitable generalized directional differential operator \eg directional subderivative \cite{Clarke2008-nonsmooth-analys}.
In this case, a closed-form expression for a Markov policy $\mu$ may not be possible altogether.
Therefore, the case of a general CLF is worth addressing.
One immediate requirement would be to assert that $\U_\X \equiv \U, \U$ compact for any $\X$.
But there is still a set of relatively strong assumptions (yet not extraordinary in the analyses and stabilization of nonlinear stochastic systems) in Definition \ref{dfn:stoch-L-pair}.
The key problem may be intuitively explained as follows.
When the Markov policy is sampled, there is no a priori bound on the mean Lyapunov function between the nodes.
Another problem is that the mean Lyapunov function and the mean state might be unrelated due to an arbitrarily large Jensen's gap.
A way to remedy all these problems is to consider noise models different to the Brownian motion, specifically, bounded ones \ie those whose realizations are bounded \as
Such models are known, and for a good review, the reader may refer to \cite{Domingo2020-bounded-stochastic-processes}.
We briefly mention these for the sake of completeness.
The easiest way to achieve bounded noise is to apply a saturation function to $W_t$.
Such is the case of the sine-Wiener process $Z_t = \sin \left ( \sqrt{\tfrac {2} {\tau_a}} W_t\right)$ with an autocorrelation time parameter $\tau_a$.
Substituting $W_t$ in the system SDE \eqref{eqn:sys} for such a noise would require extra attention to the existence and uniqueness of strong solutions, though.
Another way is to augment the system description with a dynamical noise model.
Thus, the overall description would read:
\begin{equation}
	\label{eqn:sys-bounded-noise}
	\begin{aligned}
		& \diff X_t = f(X_t, U_t) \diff t + \sigma (X_t, U_t) Z_t \diff t, \PP{X_0 = x_0} = 1, \\
		& \diff Z_t = \zeta(Z_t) \diff t + \eta(Z_t) \diff W_t,	
	\end{aligned}
\end{equation}
where $\{Z_t\}$ is the noise process with an internal model (SDE) described by the drift function $\zeta$ and diffusion function $\eta$.
Particular ways to construct such an SDE include the following \cite{Domingo2020-bounded-stochastic-processes}:
\begin{itemize}
	\item The Doering-Cai-Lin (DCL) noise \\
	\begin{equation}
	\label{eqn:DCL}
		\diff Z_t = - \tfrac 1 \theta Z_t \diff t + \sqrt{\tfrac {1 - Z_t^2 }{\theta (\gamma + 1 )}} \diff W_t,
	\end{equation}
	with parameters $\gamma>-1, \theta>0$;
	\item The Tsallis-Stariolo-Borland (TSB) noise \\
	\begin{equation}
		\label{eqn:TSB}
		\diff Z_t = - \tfrac 1 \theta \tfrac{Z_t}{1 - Z_t^2} \diff t + \sqrt{\tfrac{1 - q}{\theta}} \diff W_t,
	\end{equation}
	with $\theta > 0, \spc q < 1$ parameters;
	\item Kessler–S{\o}rensen (KS) noise \\
	\begin{equation}
		\label{eqn:KS}
		\diff Z_t = - \tfrac {\vartheta} {\pi \theta} \tan \left ( {\tfrac \pi 2 Z_t} \right ) \diff t + \tfrac {2} {\pi \sqrt{\theta (\gamma + 1 )}}  \diff W_t,
	\end{equation}   
	with $\theta > 0, \gamma \ge 0, \vartheta = \tfrac{2\gamma+1}{\gamma+1}$ parameters. 	
\end{itemize}
The above models essentially design the drift and/or diffusion functions so as to confine strong solutions to stay within $(-1,1)$ (component-wise) \as (the unitary bound is chosen for simplicity and may be adjusted according to the application).
It should be noted that the corresponding functions $\zeta, \eta$ do not satisfy Lipschitz conditions in the standard way.
Nevertheless, existence and uniqueness of strong solutions can be ensured \cite{Domingo2020-bounded-stochastic-processes}.
So, for instance, in the case of the TSB noise, the drift is at least locally Lipschitz.
This fact, together with non-reachability of the boundaries $-1,1$ (which can be shown) resolves the strong solution question.

Summarizing, we can generalize Theorem \ref{thm:stoch-stab} to systems of the kind \eqref{eqn:sys-bounded-noise} with bounded noise, and general CLFs without a closed-form Markov policy $\mu$.
The proof would go along the lines of \cite{Clarke1997-stabilization,Osinenko2018-pract-stabilization,Osinenko2019-rob-pract-stab} and require only technical modifications.
Some aspects would get simpler compared to Theorem \ref{thm:stoch-stab} as there are now only Lebesgue integrals (the It\=o integral would be ``confined'' into the inner noise model).
Subsequently, provided the decay condition on the mean Lyapunov function, the trajectories would be bounded \as (since the driving noise is now bounded \as).
The mean trajectory bound within a time interval would read $\E{X_t - X_{k \delta}} \le \bar f \delta + \bar \sigma \delta$ instead of $\E{X_t - X_{k \delta}} \le \bar f \delta$ etc.
If the CLF were non-smooth, we would need to determine control actions in a suitable way.
This can be done by various methods, such as steepest descent, Dini aiming, or inf-convolutions \cite{Braun2017-SH-stabilization-Dini-aim}.
All these methods require performing optimization at each time step.
In fact, a generalization of Theorem \ref{thm:stoch-stab} for systems of the kind \eqref{eqn:sys-bounded-noise} applies also in the case when the said optimization is non-exact \cite{Osinenko2018-pract-stabilization,Osinenko2019-rob-pract-stab}.
A comment should be made about the generator $\A$.
When dealing with systems of the kind \eqref{eqn:sys-bounded-noise}, the required CLF needs not to account for the diffusion function $\sigma$ to achieve stabilization in the sense of Definition \ref{dfn:conv-mean}, which already entails a best possible bound that the mean trajectory converge into.
We may start with a noiseless system $\dot x = f(x, u)$ and derive a CLF for it first.
Then, we may apply \ref{crl:determ2stoch} if the CLF is smooth.
If it is not, the generator $\A$ is not well-defined.
For the \SH analysis, remarkably, this does not pose such serious difficulties, as one might think.
Instead of the condition of the kind \eqref{eqn:AL-MVT-bounds}, which requires smoothness, one can work with a variant of Taylor series for non-smooth functions.
So, for instance, in the case of inf-convolution control, such a condition reads:
\begin{equation}
	\label{eqn:nonsmooth-Taylor}
	\begin{aligned}
		& \forall \eps > 0, x, y \in \R^n \\
		& L_\beta(x + \eps y) \le L_{\beta}(x) + \eps \langle \nu_\beta(x) , y \rangle + \tfrac{\eps^2 \|y\|^2}{2\beta^2},
	\end{aligned}
\end{equation}
where $v_\beta(x), \beta \in (0,1)$ is a proximal subgradient defined via:
\begin{align}
	\label{eqn:subgrad-beta}
	& v_\beta(x) := \tfrac{x - y_\beta(x)}{\beta^2}, \\
	\label{eqn:inf-conv-argmin}
	& y_\beta(x) := \argmin_{y \in \R^n} \left( L(y) + \tfrac{1}{2\beta^2} \| y-x \|^2 \right),
\end{align}
and $L_\beta$ is the inf-convolution of $L$, used to approximate $L$ as follows:
\begin{equation}
	\label{eqn:inf-conv}
	L_\beta(x) := \min_{y \in \R^n} \left( L(y) + \tfrac{1}{2\beta^2} \| y-x \|^2 \right).	
\end{equation}
An important property of $v_\beta(x)$ is that it also happens to be a proximal subgradient of $L$ itself.
Using this along with \eqref{eqn:nonsmooth-CLF-decay}, and the fact that, for any proximal subgradient $v$ of $L$ at $x \in \R^n$ and $y \in \R^n$, it holds that
\begin{equation}
	\label{eqn:subdiff-and-Dini}
	\langle v, y \rangle \le \D_v L(x),
\end{equation} 
allows one to determine practically stabilizing controls actions via
\begin{equation}
	\label{eqn:inf-conv-ctrl}
	\mu(x) := \argmin_{u \in \U^\beta_\X} \langle v_\beta(x), f(v_\beta(x), u) \rangle,	
\end{equation}
where $\U^\beta_\X$ is compact and depends only on $\X, \beta$.
In the respective \SH-analysis, the essential fragment is exactly the treatment of the kind of a Taylor series \eqref{eqn:nonsmooth-Taylor}, expressed for the actual system dynamics \eqref{eqn:sys-bounded-noise}.
So, for any $t \in [k \delta, (k+1)\delta], k \in \N_{+0}, \Delta t := t - k\delta$, one has
\begin{equation}
	\label{eqn:nonsmooth-Taylor-actual}
	\begin{aligned}
		& \E{L_t - L_{k \delta}} \le \E{\langle v_\beta(X_{k \delta}) , F_k \rangle + \tfrac{\|F_k\|^2}{2 \beta^2}},
	\end{aligned}
\end{equation}
where
\begin{equation}
	\label{eqn:rhs-sampled}
	F_k := \underbrace{\int \limits_{k \delta}^{t} f(X_\tau, \mu(X_{k \delta})) \diff\tau}_{=:\Phi_{k \delta, t}} + \underbrace{\int \limits_{k \delta}^{t} \sigma(X_\tau, \mu(X_{k \delta}) Z_\tau \diff\tau}_{=:\Sigma_{k \delta, t}}.
\end{equation}
Re-expressing $\Phi_{k \delta, t}$ as
\begin{equation}
	\label{eqn:Fk-pivot}
	\footnotesize{	
	\Delta t f(X_{k \delta}, \mu(X_{k \delta})) + \underbrace{\int \limits_{k \delta}^{t} \left( f(X_\tau, \mu(X_{k \delta})) - f(X_{k \delta}, \mu(X_{k \delta})) \right) \diff\tau}_{=: A_{k\delta, t}}
	}	
\end{equation}
converts \eqref{eqn:nonsmooth-Taylor-actual} into
\begin{equation}
	\label{eqn:nonsmooth-Taylor-decay-term}
	\begin{aligned}	
		\E{L_t - L_{k \delta}} \le & \Delta t \E{\langle v_\beta(X_{k \delta}) , f(X_{k \delta}, \mu(X_{k \delta})) \rangle} \\
		& + \E{\langle v_\beta(X_{k \delta}), A_{k\delta, t} \rangle }\\
		& + \E{\langle v_\beta(X_{k \delta}) , \Sigma_{k \delta, t} \rangle} + \E{\tfrac{ \|F_k\|^2}{2 \beta^2}} \\
	\end{aligned}
\end{equation}
By the Cauchy-Schwarz inequality,
\begin{equation}
	\label{eqn:bound_scal_v_Phi}
	\E{\langle v_\beta(X_{k \delta}) , A_{k \delta, t} \rangle} \le \sqrt{\E{ \nrm{ v_\beta(X_{k \delta}) }^2 }} \sqrt{\E{ \nrm{A_{k \delta, t}}^2 }}.
\end{equation}
The term $\E{ \nrm{ v_\beta(X_{k \delta}) }^2 }$ can be bounded exploiting the fact that the driving noise is bounded and so $X_{k \delta}$ is bounded \as (in other words, interpreting $R^*$ as an \as overshoot bound), although the respective bound will depend on $\beta$.
Next, observe:
\[
	\nrm{A_{k \delta, t}} \le \bigo{\Delta t^2}.
\]
Here, either one can exploit an \as overshoot bound $R^*$, or exploit a growth assumption like $\nrm{f(x, \mu(x))} \le (a_0 + \nrm{x})^{a_1}, a_{0,1}>0$, or just $\nrm{f(x, \mu(x))} \le \nrm{x}^a, a>0$ if $f(0, 0)=0 \land \mu(0) = 0$.
Observe that $\E{\nrm{X}^a}$ is \as bounded since $\nrm{X}$ is \as bounded by the virtue of the noise model.
However, $\E{\nrm{X}^a}$ might not be related to $\E{\nrm{X}}$ in a straightforward way.
In any case, we have:
\[
	\E{\langle v_\beta(X_{k \delta}) , A_{k \delta, t} \rangle} \le \bigo{\Delta t^2}.
\]
The following bound is due to the Cauchy-Schwarz inequality: 
\[
	\E{ \langle v_\beta(X_{k \delta}) , \Sigma_{k \delta, t} \rangle} \le \sqrt{\E{ \nrm{ v_\beta(X_{k \delta}) }^2 }} \sqrt{\E{ \nrm{\Sigma_{k \delta, t}}^2 }}.
\]
The first factor can be bounded the same way as in the case of \eqref{eqn:bound_scal_v_Phi}.
Further, we have a bound (assuming the noise is at most $1$ in norm):
\[
	\nrm{\Sigma_{k \delta, t}} \le \Delta t \sup_{\tau \in [k \delta, t]} \nrm{ \sigma(X_\tau, \mu(X_\tau)) }.
\]
Finally and combined with the previous bounds, getting the correct pivot point $v_{\beta}(X_{k \delta})$, we may use the similar token as in \cite{Osinenko2018-pract-stabilization}:
\begin{equation}
	\label{eqn:nonsmooth-Taylor-decay-term-corr}
	\begin{aligned}	
		\E{L_t - L_{k \delta}} \le & \Delta t \E{\langle v_{\beta}(X_{k \delta}) , f(v_\beta(X_{k \delta}), \mu(X_{k \delta})) \rangle} \\
		& + C_1 \beta \Delta t + C_2 \frac{\Delta t^2}{\beta} + C_3 \frac{\Delta t}{\beta} + C_4 \frac{\Delta t^2}{2 \beta^2}, \\
	\end{aligned}
\end{equation}
where $C_1, C_2, C_3, C_4>0$ are constants.
Observe $- \E{\alpha_3(X_{k \delta})} \le - \alpha_3\left(\E{\nrm{X_{k \delta}}}\right)$ by the Jensen's inequality if $\alpha_3$ is convex.
So, $\Delta t \E{\langle v_{\beta}(X_{k \delta}) , f(v_\beta(X_{k \delta}), \mu(X_{k \delta})) \rangle}$ yields a suitable decay term.
Similarly to Corollary \ref{crl:determ2stoch}, either we have to assume uniform boundedness of $\sigma$, or that it possess a growth rate in $\nrm{x}$ not dominating over $\alpha_3$ anywhere except for vicinity of the origin.
The latter may be affected by $\beta$ due to the effect of combination of the terms $C_1 \beta \Delta t, C_3 \frac{\Delta t}{\beta}$ in \eqref{eqn:nonsmooth-Taylor-decay-term-corr}.
Now consider:
\begin{thm}
	\label{thm:stoch-stab-nonsmooth}
	Consider a stochastic system
	\begin{equation*}
		\begin{aligned}
			& \diff X_t = f(X_t, U_t) \diff t + \sigma (X_t, U_t) Z_t \diff t, \PP{X_0 = x_0} = 1, \\
			& \diff Z_t = \zeta(Z_t) \diff t + \eta(Z_t) \diff W_t	
		\end{aligned}
	\end{equation*}
	with bounded noise $Z_t$.
	Suppose there exists a (in general non-smooth) CLF $L$ with a decay condition
	\begin{equation*}
	\begin{aligned}
		& \forall \text{ compact } \X \spc \exists \text{ compact } \U_\X \\
		& \forall x \in \X \spc \inf_{\nu \in \co(f(x, \U_\X))} \D_\nu L(x) \le - \alpha_3(\nrm{x}).	
	\end{aligned}
	\end{equation*}	
	If $\sigma$ is uniformly bounded or does not dominate $\alpha_3$ in growth anywhere except for vicinity of the origin, there exists a map $\mu: \X \ra \U_\X$ that, applied in \SH mode, practically stochastically stabilizes the system in mean with $\rho$ (see Definition \ref{dfn:pract-stab}) depending on noise properties and $\beta$ of \eqref{eqn:subgrad-beta}.
\end{thm}


Regarding robustness properties, the presented derivations may be generalized to the case of uncertainty.
Whereas the system and actuator uncertainties are relatively easy to address, the major problem is measurement noise \cite{Clarke2011-discont-stabilization}.
Still, the method of inf-convolutions described above can be shown robust in this regard \cite{Osinenko2019-rob-pract-stab}.
Merging this with Theorem \ref{thm:stoch-stab-nonsmooth} could achieve the desired robustness, but it is left for future work.
\addtolength{\textheight}{-12cm}   


\vspace{-0.25em}

\bibliography{bib/adversarial,
bib/analysis,
bib/automated-reasoning,
bib/computable,
bib/constr-math,
bib/ctrl-history,
bib/diff-games,
bib/discont-DE,
bib/Filippov-sol,
bib/formal-ctrl,
bib/lin-ctrl,
bib/logic,
bib/model-reduction,
bib/MPC,
bib/nonlin-ctrl,
bib/nonsmooth-analysis,
bib/opt-ctrl,
bib/Osinenko,
bib/perturb-thr,
bib/selector-misc,
bib/semiconcave,
bib/sensitivity,
bib/set-thr,
bib/sliding-mode,
bib/soft,
bib/stabilization,
bib/stochastic,
bib/topology,
bib/uncertainty,
bib/verified-integration,
bib/viability,
bib/metric-spaces,
bib/DP,
bib/constructing-LFs,
bib/Kalman-filter} 

\begin{thebibliography}{10}
\providecommand{\url}[1]{#1}
\csname url@samestyle\endcsname
\providecommand{\newblock}{\relax}
\providecommand{\bibinfo}[2]{#2}
\providecommand{\BIBentrySTDinterwordspacing}{\spaceskip=0pt\relax}
\providecommand{\BIBentryALTinterwordstretchfactor}{4}
\providecommand{\BIBentryALTinterwordspacing}{\spaceskip=\fontdimen2\font plus
\BIBentryALTinterwordstretchfactor\fontdimen3\font minus
  \fontdimen4\font\relax}
\providecommand{\BIBforeignlanguage}[2]{{%
\expandafter\ifx\csname l@#1\endcsname\relax
\typeout{** WARNING: IEEEtran.bst: No hyphenation pattern has been}%
\typeout{** loaded for the language `#1'. Using the pattern for}%
\typeout{** the default language instead.}%
\else
\language=\csname l@#1\endcsname
\fi
#2}}
\providecommand{\BIBdecl}{\relax}
\BIBdecl

\bibitem{khasminskii201-stochastic}
R.~Khasminskii and G.~Milstein, \emph{Stochastic Stability of Differential
  Equations}, ser. Stochastic Modelling and Applied Probability.\hskip 1em plus
  0.5em minus 0.4em\relax Springer, 2011.

\bibitem{Kushner1965stochastic-stability}
H.~J. Kushner, ``On the stability of stochastic dynamical systems,''
  \emph{Proc. National Academy of Sciences of USA}, vol.~53, no.~1, pp. 8--12,
  1965.

\bibitem{mao1991-stability}
X.~Mao, \emph{Stability of Stochastic Differential Equations with Respect to
  Semimartingales}, ser. Pitman research notes in mathematics series.\hskip 1em
  plus 0.5em minus 0.4em\relax Longman Scientific \& Technical, 1991.

\bibitem{Deng2001-stochastic-stab-noise}
{Hua Deng}, M.~{Krstic}, and R.~J. {Williams}, ``Stabilization of stochastic
  nonlinear systems driven by noise of unknown covariance,'' \emph{IEEE Tran.
  on Autom. Control}, vol.~46, no.~8, pp. 1237--1253, 2001.

\bibitem{Khalil1996-nonlin-sys}
H.~Khalil, \emph{Nonlinear {S}ystems}.\hskip 1em plus 0.5em minus 0.4em\relax
  Prentice-Hall. 2nd edition, 1996.

\bibitem{McAllister2003-stochastic}
R.~D. McAllister and J.~B. Rawlings, ``Stochastic lyapunov functions and
  asymptotic stability in probability,'' Tech. Rep., 08 2020.

\bibitem{Liu2008-stochastic}
S.-J. Liu, J.-f. Zhang, and Z.-p. Jiang, ``A notion of stochastic
  input-to-state stability and its application to stability of cascaded
  stochastic nonlinear systems,'' \emph{Acta Mathematicae Applicatae Sinica},
  vol.~24, pp. 141--156, 03 2008.

\bibitem{Liu2011-stochastic}
L.~Liu and X.-J. Xie, ``Output-feedback stabilization for stochastic high-order
  nonlinear systems with time-varying delay,'' \emph{Automatica}, vol.~47,
  no.~12, pp. 2772--2779, 2011.

\bibitem{Li2017-Stochastic}
H.~{Li}, L.~{Bai}, Q.~{Zhou}, R.~{Lu}, and L.~{Wang}, ``Adaptive fuzzy control
  of stochastic nonstrict-feedback nonlinear systems with input saturation,''
  \emph{IEEE Tran. on Syst., Man, and Cybernetics: Systems}, vol.~47, no.~8,
  pp. 2185--2197, 2017.

\bibitem{Wu2013-stochastic}
Z.~{Wu}, M.~{Cui}, P.~{Shi}, and H.~R. {Karimi}, ``Stability of stochastic
  nonlinear systems with state-dependent switching,'' \emph{IEEE Tran. on
  Autom. Control}, vol.~58, no.~8, pp. 1904--1918, 2013.

\bibitem{Teel2014-stochastic-stability}
A.~R. Teel, A.~Subbaraman, and A.~Sferlazza, ``Stability analysis for
  stochastic hybrid systems: A survey,'' \emph{Automatica}, vol.~50, no.~10,
  pp. 2435--2456, 2014.

\bibitem{caraballo2015-practical-stochastic}
T.~Caraballo, M.~A. Hammami, and L.~Mchiri, ``On the practical global uniform
  asymptotic stability of stochastic differential equations,'' 2015.

\bibitem{Qin2020-stochastic}
Y.~Qin, M.~Cao, and B.~D.~O. Anderson, ``Lyapunov criterion for stochastic
  systems and its applications in distributed computation,'' \emph{IEEE Tran.
  on Autom. Control}, vol.~65, no.~2, pp. 546--560, 2020.

\bibitem{Do2020-stochastic}
K.~Do, ``Practical asymptotic stability of stochastic systems driven by lévy
  processes and its application to control of tora systems,'' \emph{Int. J.
  Control}, pp. 1--25, 03 2020.

\bibitem{fu2016sampled}
X.~Fu, Y.~Kang, and P.~Li, ``Sampled-data stabilization for a class of
  stochastic nonlinear systems based on the approximate discrete-time models,''
  in \emph{2016 Australian Control Conference (AuCC)}.\hskip 1em plus 0.5em
  minus 0.4em\relax IEEE, 2016, pp. 258--263.

\bibitem{yu2018sampled}
P.~Yu, Y.~Kang, and Q.~Zhang, ``Sampled-data stabilization for a class of
  stochastic nonlinear systems with markovian switching based on the
  approximate discrete-time models,'' in \emph{2018 Australian \& New Zealand
  Control Conference (ANZCC)}.\hskip 1em plus 0.5em minus 0.4em\relax IEEE,
  2018, pp. 413--418.

\bibitem{Clarke2011-discont-stabilization}
F.~Clarke, ``{L}yapunov functions and discontinuous stabilizing feedback,''
  \emph{{A}nnual {R}eviews in {C}ontrol}, vol.~35, no.~1, pp. 13--33, 2011.

\bibitem{mao2007stochastic}
X.~Mao, \emph{Stochastic differential equations and applications}.\hskip 1em
  plus 0.5em minus 0.4em\relax Elsevier, 2007.

\bibitem{Reif1999-EKF-stab}
K.~Reif, S.~Gunther, E.~Yaz, and R.~Unbehauen, ``{S}tochastic stability of the
  discrete-time extended {K}alman filter,'' \emph{IEEE Transactions on
  Automatic Control}, vol.~44, no.~4, pp. 714--728, 1999.

\bibitem{li2021stochastic}
W.~Li and M.~Krstic, ``Stochastic nonlinear prescribed-time stabilization and
  inverse optimality,'' \emph{IEEE Tran. on Autom. Control}, 2021.

\bibitem{lan2017global}
Q.~Lan and S.~Li, ``Global output-feedback stabilization for a class of
  stochastic nonlinear systems via sampled-data control,'' \emph{Int. J. Robust
  and Nonlin. Control}, vol.~27, no.~17, pp. 3643--3658, 2017.

\bibitem{yang2011mean}
S.~Yang, B.~Shi, and M.~Li, ``Mean square stability of impulsive stochastic
  differential systems,'' \emph{Int. J. Differential Equations}, vol. 2011,
  2011.

\bibitem{Lin1991-stabilization-bounded-controls}
Y.~Lin and E.~D. Sontag, ``A universal formula for stabilization with bounded
  controls,'' \emph{Syst. \& Control Letters}, vol.~16, no.~6, pp. 393--397,
  1991.

\bibitem{Clarke1997-stabilization}
F.~Clarke, Y.~Ledyaev, E.~Sontag, and A.~Subbotin, ``Asymptotic controllability
  implies feedback stabilization,'' \emph{IEEE Tran. on Autom. Control},
  vol.~42, no.~10, pp. 1394--1407, 1997.

\bibitem{oksendal2003stochastic}
B.~{\O}ksendal, ``Stochastic differential equations,'' in \emph{Stochastic
  Differential Equations}.\hskip 1em plus 0.5em minus 0.4em\relax Springer,
  2003, pp. 65--84.

\bibitem{Sontag1989-formula}
E.~Sontag, ``A "universal" construction of artstein's theorem on nonlinear
  stabilization,'' \emph{Syst. \& Control Letters}, vol.~13, no.~2, pp.
  117--123, 1989.

\bibitem{Florchinger1994-Lyapunov-like-stochastic-stability}
P.~{Florchinger}, ``Lyapunov-like techniques for stochastic stability,'' in
  \emph{IEEE CDC}, vol.~2, 1994, pp. 1145--1150.

\bibitem{Clarke2008-nonsmooth-analys}
F.~Clarke, Y.~Ledyaev, R.~Stern, and P.~Wolenski, \emph{Nonsmooth Analysis and
  Control Theory}.\hskip 1em plus 0.5em minus 0.4em\relax Springer, 2008, vol.
  178.

\bibitem{Domingo2020-bounded-stochastic-processes}
D.~Domingo, A.~d’Onofrio, and F.~Flandoli, ``Properties of bounded stochastic
  processes employed in biophysics,'' \emph{Stochastic Analysis and Appl.},
  vol.~38, no.~2, pp. 277--306, 2020.

\bibitem{Osinenko2018-pract-stabilization}
P.~Osinenko, L.~Beckenbach, and S.~Streif, ``Practical sample-and-hold
  stabilization of nonlinear systems under approximate optimizers,''
  \emph{Control Syst. Letters}, vol.~2, no.~4, pp. 569--574, 2018.

\bibitem{Osinenko2019-rob-pract-stab}
P.~Schmidt, P.~Osinenko, and S.~Streif, ``On inf-convolution-based robust
  practical stabilization under computational uncertainty,'' \emph{IEEE Trans.
  Autom. Control}, 2021.

\bibitem{Braun2017-SH-stabilization-Dini-aim}
P.~Braun, L.~Gr{\"u}ne, and C.~Kellett, ``Feedback design using nonsmooth
  control {L}yapunov functions: A numerical case study for the nonholonomic
  integrator,'' in \emph{IEEE CDC}, 2017.

\end{thebibliography}
\bibliographystyle{IEEEtran}

%

\end{document}